\documentclass[a4paper,11pt]{amsart}

\usepackage{latexsym}
\usepackage{amssymb}
\usepackage{amsthm}
\usepackage{amscd}
\usepackage{amsmath}
\usepackage[dvipdfmx]{graphicx}
\usepackage[all]{xy}
\usepackage{multicol}

\newtheorem{Theorem}{Theorem}[section]
\newtheorem{Lemma}[Theorem]{Lemma}

\newtheorem{Proposition}[Theorem]{Proposition}
\newtheorem{Remark}[Theorem]{Remark}

\newtheorem{Example}[Theorem]{Example}

\newtheorem{Definition}[Theorem]{Definition}

\newtheorem{Question}[Theorem]{Question}

\newtheorem{Observation}[Theorem]{Observation}
\newtheorem{Notation}[Theorem]{Notation}

  \makeatletter
    
    \@addtoreset{equation}{section}
  \makeatother

\def\fka{{\frak a}}
\def\fkb{{\frak b}}
\def\fkc{{\frak c}}

\def\fkm{{\frak m}}

\def\opn#1#2{\def#1{\operatorname{#2}}}

\opn\Spec{Spec}
\opn\Supp{Supp}
\opn\supp{supp}
\opn\Max{Max}
\opn\max{max}
\opn\Min{Min}
\opn\min{min}
\opn\Ass{Ass}
\opn\Assh{Assh}
\opn\Ann{Ann}
\opn\depth{depth}
\opn\rank{rank}

\opn\Mat{Mat}

\opn\Tot{Tot}

\opn\Sym{Sym}
\def\Rees{{\mathcal R}}

\opn\ord{ord}

\opn\div{div}
\opn\Div{Div}
\opn\cl{cl}
\opn\Cl{Cl}

\opn\Ker{Ker}
\opn\Coker{Coker}
\opn\Im{Im}
\opn\Hom{Hom}
\opn\Tor{Tor}
\opn\Ext{Ext}
\opn\End{End}
\opn\Fitt{Fitt}
\opn\Aut{Aut}
\opn\id{id}

\opn\nat{nat}
\opn\pff{pf}
\opn\Pf{Pf}
\opn\GL{GL}
\opn\SL{SL}
\opn\G{G}
\opn\E{E}
\opn\H{H}
\opn\M{M}
\opn\mod{mod}
\opn\ord{ord}
\opn\det{det}
\opn\Soc{Soc}

\opn\chara{char}
\opn\length{\ell}
\opn\pd{pd}
\opn\rk{rk}
\opn\projdim{proj\,dim}
\opn\injdim{inj\,dim}
\opn\rank{rank}
\opn\depth{depth}
\opn\grade{grade}
\opn\height{ht}
\opn\embdim{emb\,dim}
\opn\codim{codim}

\renewcommand{\tilde}{\widetilde}
\renewcommand{\bar}{\overline}

\title{Constructing indecomposable integrally closed modules over a two-dimensional regular local ring} 
\author{Futoshi Hayasaka}
\address{Department of Environmental and Mathematical Sciences, Okayama University, 
3-1-1 Tsushimanaka, Kita-ku, Okayama, 700-8530, JAPAN}
\email{hayasaka@okayama-u.ac.jp}

\begin{document}

\maketitle

\begin{abstract}
In this article, we construct integrally closed modules of rank two 
over a two-dimensional regular local ring. The modules are explicitly constructed from 
a given complete monomial ideal with respect to a regular system of parameters. 
Then we investigate their indecomposability. 
As a consequence, we have a large class of indecomposable 
integrally closed modules whose Fitting ideal is not simple. 
This gives an answer to Kodiyalam's question. 
\end{abstract}

\section{Introduction}

The theory of complete (integrally closed) ideals in a regular local ring of dimension two 
was developed by Zariski in \cite{Za} and in \cite[Appendix 5]{ZS}. 
Zariski proved two structure theorems. The first main result is 
the product theorem. It asserts that 
the product of any two complete ideals in a two-dimensional regular local ring is again a complete ideal. 
The second main result is the unique factorization theorem. It asserts that any non-zero complete ideal
in a two-dimensional regular local ring can be expressed uniquely (except for ordering) 
as a product of simple complete ideals. 
Here an ideal is simple if it cannot be expressed as a product of two proper ideals. 
Since the classic work of Zariski, the theory has been attracting interest and 
has been generalized to more general situations. See for instance the papers \cite{Cut, HuSa, Li}. 
Among interesting results in this direction, a generalization of Zariski's product theorem to 
finitely generated torsion-free integrally closed modules was obtained by Kodiyalam 
in \cite{Ko}. 

The notion of integral closure of modules was introduced by Rees in \cite{Re}. 
Let $A$ be a Noetherian integral domain and let $M$ be a finitely generated torsion-free $A$-module. 
The integral closure of $M$, denoted by $\bar M$, 
is defined as a set of all elements 
$f \in M_K:=M \otimes_R K$ such that $f \in MV$ 
for every discrete valuation ring $V$ of 
$K$ containing $A$. Here $K$ is the quotient field of $A$, 
and $MV$ denotes the $V$-submodule of $M_K$ generated by $M$. 
The integral closure $\bar M$ is an $R$-submodule of $M_K$ containing $M$. 
The module $M$ is said to be integrally closed if $\bar M=M$. 

Let $R$ be a two-dimensional regular local ring with infinite residue field. 
Kodiyalam proved in \cite{Ko} that the product $MN$ of 
any two finitely generated torsion-free integrally closed $R$-modules $M$ and $N$ 
is again integrally closed
in the sense of Rees. Here the product $MN$ is the tensor product modulo $R$-torsion. 
Therefore, Kodiyalam's extension can be viewed as a natural generalization of Zariski's product theorem. 
Moreover, he proved that a certain Fitting ideal associated with an integrally closed module 
is again integrally closed. 
Let $F=M^{**}$ be the double $R$-dual of $M$. Then $F$ is free and 
it canonically contains $M$ with the quotient $F/M$ of finite length.
Thus, one can define the ideal $I(M)$ of $M$ as $I(M)=\Fitt_0(F/M)$. 
With this notation, Kodiyalam proved the following. 

\begin{Theorem}\label{kodi} $(${\rm Kodiyalam} \cite[Theorems 5.4, 5.7]{Ko}$)$
Let $(R, \fkm)$ be a two-dimensional regular local ring with the maximal ideal $\fkm$, 
infinite residue field $R/\fkm$. 
For a finitely generated torsion-free $R$-module $M$, we have the following. 
\begin{enumerate}
\item Suppose that $M$ is integrally closed. Then the ideal $I(M)$ of $M$ is again integrally closed. 
Furthermore, taking integral closure commutes with taking the ideal, that is, $I(\bar M)=\bar{I(M)}$. 
\item Suppose that $M$ has no free direct summand and $\bar{I(M)}$ is simple. 
Then $\bar M$ is an indecomposable $R$-module. 
In particular, there exist indecomposable integrally closed $R$-modules of arbitrary rank. 
\end{enumerate}
\end{Theorem}

A motivation of this article comes from the following question which 
can be found in the last paragraph in \cite[Example 5.8]{Ko}. 

\begin{Question}\label{ques}
{\rm 
Does the converse to Theorem \ref{kodi} (2) hold in the sense that 
an indecomposable integrally closed $R$-module $M$ of rank bigger than $1$ 
have a simple complete ideal $I(M)$?
}
\end{Question}

The purpose of this article is to give an answer to Question \ref{ques} by showing that 
there are numerous counterexamples and is to shed some light on a theory of integrally closed modules. 
In fact, we prove a stronger result 
which shows the ubiquity of indecomposable integrally closed modules of rank $2$ with the monomial 
Fitting ideal. Our results can be summarized as follows. 

\begin{Theorem}\label{main}
Let $(R, \fkm)$ be a two-dimensional regular local ring with the maximal ideal $\fkm$, 
infinite residue field $R/\fkm$. 
Let $x, y$ be a regular system of parameters for $R$ and let $I$ be 
an $\fkm$-primary complete monomial ideal 
with respect to $x, y$. Suppose that either 
\begin{enumerate}
\item $\ord(I) \geq 3$, or
\item $\ord(I)=2$ and $xy \notin I$
\end{enumerate}
is satisfied. Then there exists a finitely generated torsion-free indecomposable 
integrally closed $R$-module $M$ of rank $2$ 
with $I(M)=I$.  
\end{Theorem}

As a direct consequence, we have a large class of counterexamples to Question \ref{ques}. 
Indeed, if we consider a non-simple complete monomial ideal $I$ with $\ord(I)\geq 3$, 
e.g. $I=\fkm^{r}$ where $r \geq 3$ as the simplest case, 
then Theorem \ref{main} shows that one can find such a counterexample $M$ with $I(M)=I$. 
These modules are obtained quite explicitly 
from a given complete monomial ideal.

This article is organized as follows. 
In section 2, we collect basic facts from \cite{Ko} on integrally closed modules 
over a two-dimensional regular local ring. We also 
fix our notations we will use throughout this article. 
In section 3, we introduce a certain module of rank $2$, denoted by $M_k$ or $M_k(I)$, 
associated to a given monomial ideal $I$ with respect to a regular system of parameters $x, y$ 
and an integer $k$; see Definition \ref{3.2}. 
The modules $M_{k}$ play a central role in this article. 
A crucial point is that, for any $1 \leq k \leq r-1$, 
the associated module $M_{k}$ is integrally closed with $I(M_{k})=I$ if $I$ is complete; 
see Theorem \ref{3.6}. 
In section 4, we investigate the indecomposability of the modules $M_k$ 
when a given monomial ideal $I$ is integrally closed with order at least $3$. 
One important fact is that the associated module $M_k$ has another Fitting ideal 
$I_1(M_k)$ of order $1$; see Observation \ref{4.1}. Together with Zariski's factorization theorem, 
we can readily get a class of indecomposable integrally closed 
module whose Fitting ideal is $I$, if a given monomial complete ideal $I$ has 
no simple factor of order $1$; see Theorem \ref{4.2}. 
When a given complete monomial ideal $I$ has a simple factor of order $1$, we divide the problem into two cases. 
One case is when a given ideal $I$ does not have a simple factor of the form $(x, y^{\ell})$ 
for some $1 \leq \ell \leq r-1$. 
In this case, the problem can be reduced to particular cases; see Observation \ref{4.3}, and 
then we have Theorem \ref{4.7}. 
The other case is when $I$ is of the form $I=(x,y)(x,y^2) \cdots (x,y^{r-1}) \fkb$
where $\fkb=(x^{\alpha}, y)$ or $\fkb=(x,y^{\beta})$.
In this case, we consider the next modules $M_{r}$ and $M_{r+1}$, and then we have 
Theorems \ref{4.9} and \ref{4.11}. 
In section 5, we complete a proof of Theorem \ref{main} and give some examples to 
illustrate our results. 

Throughout this article, let $(R, \fkm)$ be a two-dimensional regular local ring with the maximal ideal 
$\fkm$, infinite residue field $R/\fkm$. Let $K$ be the quotient field of $R$. For an ideal $\fka$ in $R$, 
the order of 
$\fka$ will be denoted by $\ord(\fka)=\max \{n \mid \fka \subset \fkm^{n} \}$. 
For an $R$-module $L$, the notations $\rank_{R}(L)$ and $\mu_{R}(L)$ will denote respectively 
the rank and the minimal number of generators of $L$. The notation $\ell_{R}(\ast)$ will denote 
the length function on $R$-modules. We will use both the term ``integrally closed'' and the classical one ``complete'' for ideals.

\section{Preliminaries}

In this section, we collect some basic facts from \cite{Ko} on integrally closed modules over $R$. 
See also \cite{Hu, HuSw, ZS} for the details on a theory of complete ideals in $R$ and 
\cite{Re} for the details on a theory of integral closure of modules.  

Let $M$ be a finitely generated torsion-free $R$-module. 
We denote $M^*:=\Hom_R(M, R)$ the $R$-dual of $M$, and let $F:=M^{**}$ be the double dual of $M$. 
Then $F$ is $R$-free and it canonically contains $M$ with the quotient $F/M$ of finite length. 
Indeed, one can see that if $M$ is contained in a free $R$-module $G$ 
with the quotient $G/M$ of finite length, 
then there is a unique $R$-linear isomorphism $\varphi: F \to G$ such that 
the restriction $\varphi|_M$ is identity on $M$ (\cite[Proposition 2.1]{Ko}). 
Thus, the two quotient modules $F/M \cong G/M$ are isomorphic as $R$-modules. 
In fact, $F/M$ is isomorphic to the $1$st local cohomology module $H^1_{\fkm}(M)$ of $M$ 
with respect to $\fkm$. Therefore, one can define Fitting ideals associated to $M$ as follows: 
\begin{align*}
I(M) &=\Fitt_0(F/M)  \\
I_1(M) &= \Fitt_1(F/M)  
\end{align*}

Let $M_K=M \otimes_R K$. A subring $S$ of $K$ containing $R$ is called birational overring of $R$. For 
such a ring $S$, let $MS:=\Im (M \otimes_R S \to M_K)$ denote an 
$S$-submodule of $M_K$ generated by $M$, 
which is isomorphic to the tensor product $M \otimes_{R} S$ modulo $S$-torsion. 
Then the integral closure $\bar M$ of $M$ is defined as  
$$\bar M=\{f \in M_K \mid f \in MV \ \text{for every discrete valuation ring} \ 
V \ \text{with} \ R \subset V \subset K \}. $$
Since $R$ is a two-dimensional regular local ring, and, hence, it is normal, 
the integral closure $\bar M$ can be considered in the free module $F$, and we have the following
criteria for integral dependence of a module (see \cite{Re} and also \cite[Theorem 3.2]{Ko}).  
\begin{align*}
\bar M &= \{f \in F \mid f \in \Sym_R^1(F) \ \text{is integral over} \ \Sym_R(M) \} \\
&=\{f \in F \mid I(M+Rf) \subset \bar{I(M)}\}.
\end{align*}
These criteria imply the following property (\cite[Corollary 3.3]{Ko}): 
$$
M \ \text{is integrally closed if and only if so is} \ M_Q \ \text{for every maximal ideal} \  
Q \ \text{of} \ R.
$$ 

In Zariski's theory of complete ideals in $R$, contracted ideals play an important role. 
Kodiyalam extended this notion to modules as follows. 

\begin{Definition}\label{2.1}
Let $S$ be a birational overring of $R$. 
Then a finitely generated torsion-free $R$-module $M$ is said to be contracted from $S$, 
if the equality 
$$MS \cap F=M$$ 
holds true as submodules of $FS$. 
\end{Definition}

Here we recall some basic properties of contracted modules. 

\begin{Proposition}\label{2.2}
Let $M$ be a finitely generated torsion-free $R$-module. For any 
$x \in \fkm \setminus \fkm^2$, 
the following conditions are equivalent. 
\begin{enumerate}
\item $M$ is contracted from $S=R[ \frac{\fkm}{x} ]. $
\item The equality $\fkm M :_F x=M$ holds true. 
\item The equality $M :_F x=M:_F \fkm$ holds true.  
\end{enumerate}
Therefore, if $M$ is integrally closed, then $M$ is contracted from $S=R[\frac{\fkm}{x}]$ 
for some $x \in \fkm \setminus \fkm^2$. 
\end{Proposition}

\begin{proof}
See the proof of \cite[Propositions 2.5 and 4.3]{Ko}. 
\end{proof}

Here is another useful characterization of contracted modules. 

\begin{Proposition}\label{2.3}
Let $M$ be a finitely generated torsion-free $R$-module. Then the following conditions are equivalent. 
\begin{enumerate}
\item $M$ is contracted from $S=R[ \frac{\fkm}{x} ]$ for some $x \in \fkm \setminus \fkm^2$. 
\item The equality $\mu_R(M)=\ord(I(M))+\rank_R(M)$ holds true. 
\end{enumerate}
Moreover, when this is the case, for any $x \in \fkm \setminus \fkm^2$ such that 
$$\ell_R(R/I(M)+(x))=\ord(I(M)), $$
the module $M$ is contracted from $S=R[ \frac{\fkm}{x} ]$. 
\end{Proposition}

\begin{proof}
See \cite[Proposition 2.5]{Ko}.
\end{proof}

Consider a birational overring $S=R[\frac{\fkm}{x}]$ of $R$ 
where $x \in \fkm \setminus \fkm^2$. Then it is well-known that for any maximal ideal $Q$ of $R$, 
\begin{itemize}
\item $S_Q$ is a discrete valuation ring when $Q \nsupseteq \fkm S$
\item $S_Q$ is a two-dimensional regular local ring when $Q \supseteq \fkm S$
\end{itemize}
The two-dimensional regular local ring $S_Q$ for a maximal ideal $Q$ of $S$ containing $\fkm S$ is called a first quadratic 
transform of $R$. For an $\fkm$-primary ideal $I$ in $R$ with $\ord(I)=r$, we can write 
$IS=x^r[IS:_S x^r]. $
Then we define the ideal $I^{S}$ as 
$$I^S=IS:_S x^r$$ 
and call it a transform of $I$ in $S$. For a first quadratic transform $T:=S_Q$ of $R$, 
we also define a transform $I^{T}$ of $I$ in $T$ as 
$$I^T=I^S T. $$ 

Contracted modules have the following nice property. 

\begin{Proposition}\label{2.4}
Let $M$ be a finitely generated torsion-free $R$-module. 
Suppose that $M$ is contracted from $S=R[ \frac{\fkm}{x} ]$ for some 
$x \in \fkm \setminus \fkm^2$. 
Then the following conditions are equivalent. 
\begin{enumerate}
\item $M$ is an integrally closed $R$-module. 
\item $MS$ is an integrally closed $S$-module. 
\end{enumerate}
In particular, when this is the case, 
for any first quadratic transform $T$ of $R$, $MT$ is an integrally closed $T$-module.  
\end{Proposition}

\begin{proof}
See \cite[Proposition 4.6]{Ko}. 
\end{proof}

Therefore, for an $\fkm$-primary integrally closed ideal $I$ in $R$, 
a transform $I^T$ of $I$ in a first quadratic transform $T=S_Q$ of $R$ is also integrally closed. 
Indeed, since $I$ is complete, $IS$ and hence $I^{S}$ is integrally closed 
by Proposition \ref{2.4} so that its 
localization $I^T$ is also integrally closed. 

One of crucial points in the theory of both integrally closed ideals and modules is that 
the colength of a transform $I^{T}$ in a first quadratic transform $T$ of $R$ is less than 
the one of an $\fkm$-primary ideal $I$. Namely, for an $\fkm$-primary ideal $I$ of $R$ and a first 
quadratic transform $T$ of $R$, the inequality 
$$\ell_R(R/I) > \ell_T(T/I^T) $$
holds true (\cite[Theorem 4.5]{Ko}). 

The ideal $I(M)$ of $M$ behaves well under transforms. 
Therefore, as in the ideal case, taking a transform $MT$ improves the module $M$.  

\begin{Proposition}\label{2.5}
Let $M$ be a finitely generated torsion-free $R$-module and $T$ a first quadratic transform of $R$. Then the equality 
$$I(MT)=I(M)^T$$
holds true. 
\end{Proposition}

\begin{proof}
See \cite[Proposition 4.7]{Ko}. \end{proof}

Before closing this preliminary section, we fix notations we will use in the rest of this article. 

\begin{Notation}
Let $A$ be an arbitrary Noetherian ring and let $A^n=At_1+\dots +At_n$ be a free $A$-module of rank $n>0$ 
with free basis $t_1, \dots , t_n$. 
\begin{itemize}
\item For a submodule $L=\langle f_1, \dots , f_m \rangle$ of $A^n$ generated by $f_1, \dots , f_m$, we denote the associated matrix
$$\tilde L:=(a_{ij}) \in \Mat_{n \times m}(A)$$
where $f_j=a_{1j}t_1+\dots +a_{nj}t_n$ for $j=1, \dots , m$. 
\item Conversely, for a matrix $\varphi=(b_{ij}) \in \Mat_{n \times m}(A)$, we denote the associated submodule of $A^n$
$$\langle \varphi \rangle=\langle g_1, \dots , g_m \rangle$$
where $g_j=b_{1j}t_1+\dots +b_{nj}t_n$ for $j=1, \dots , m$.
\item For two matrices $\varphi \in \Mat_{n \times m}(A)$ and $\psi \in \Mat_{n \times m'}(A)$ with the same number of rows, we define a relation $\sim$ as 
$$\varphi \sim \psi \Leftrightarrow \langle \varphi \rangle \cong \langle \psi \rangle \ \text{as $A$-modules}$$
\item For a matrix $\varphi=(b_{ij}) \in \Mat_{n \times m}(A)$, we denote the ideal in $A$ 
generated by all the $t$-minors of $\varphi$ 
$$I_{t}(\varphi)$$
\end{itemize}
\end{Notation}

\section{Integrally closed modules of rank two}

Recall that $R$ is a two-dimensional regular local ring with the maximal ideal $\fkm$, 
infinite residue field $R/\fkm$. Throughout this section, we consider 
\begin{itemize}
\item a fixed regular system of parameters $x, y$ for $R$, that is, $\fkm=(x,y)$, and 
\item an $\fkm$-primary monomial ideal $I$ with respect to $x, y$.  
\end{itemize}
We write the monomial ideal $I$ as 
\begin{align}\label{mono}
I &=(x^{a_i}y^{b_i} \mid 0 \leq i \leq r)=(x^{a_0}, x^{a_1}y^{b_1}, \dots , x^{a_{r-1}}y^{b_{r-1}}, y^{b_r}) 
\end{align}
where 
\begin{align*}
a_0>a_1>\dots >a_{r-1}>a_r:=0 \ \text{and} \ b_0:=0<b_1<\dots <b_{r-1}<b_r. 
\end{align*}

We begin with the following. 

\begin{Lemma}\label{3.1}
For the monomial ideal $I$, we have the following. 
\begin{enumerate}
\item $\mu_R(I)=r+1$
\item $\ord(I)=\min\{ a_i+b_i \mid 0 \leq i \leq r\}$
\item $\ell_R(R/I+(x+y))=\ord(I)$
\end{enumerate}
\end{Lemma}

\begin{proof}
We show the assertion (1). Suppose that 
$$\sum_{i=0}^{r} \alpha_{i}(x^{a_{i}}y^{b_{i}})=0$$ 
where $\alpha_{i} \in R$. Then $\alpha_{0} x^{a_{0}} \in (y^{b_{1}})$ and thus $\alpha_{0} \in (y^{b_{1}})$. Similarly, $\alpha_{r} \in (x^{a_{r-1}})$ because $\alpha_{r} y^{b_{r}} \in (x^{a_{r-1}})$. 
Let $1 \leq i \leq r-1$. Since $\alpha_{i} (x^{a_{i}}y^{b_{i}}) \in (x^{a_{i-1}}, y^{b_{i+1}})$, 
it follows that 
$$\alpha_{i} \in (x^{a_{i-1}-a_{i}}, y^{b_{i+1}-b_{i}}). $$  
Thus, $\alpha_{i} \in \fkm$ for all $0 \leq i \leq r$. 
This shows that $\mu_{R}(I)=r+1$. 

The assertion (2) is easy to see. We show the assertion (3). 
Let $r_{0}=\ord(I)$. Note that for any $i \geq 0$ and $j \geq 1$, 
$$x^{i+1}y^{j-1}+x^{i}y^{j}=x^{i}y^{j-1}(x+y). $$ 
Thus, $I+(x+y)=(x^{r_{0}}, x+y)$, and we get that $\ell_{R}(R/I+(x+y))=r_{0}=\ord(I). $ \end{proof}

Here we consider the following modules associated to the monomial ideal $I$ and an integer $k$.
These play a central role in this article. 

\begin{Definition}\label{3.2}
Let $1 \leq k < b_r$ be an integer. Then we define a module $M_k$ associated to the monomial ideal $I$ and 
the integer $k$ as follows$:$
\begin{align*}
M_k &:=\left\langle
\begin{pmatrix}
x^{a_0-1} & \cdots & x^{a_i-1}y^{b_i} & \cdots & x^{a_{r-1}-1}y^{b_{r-1}} & y^k & 0 \\
0 &\cdots &  0 & \cdots & 0 & x & y^{b_r-k}
\end{pmatrix} \right\rangle \\
& \subset F:=
\left\langle
\begin{pmatrix}
1 & 0 \\
0 & 1
\end{pmatrix}
\right\rangle. 
\end{align*}
The module $M_k$ will be denoted by $M_k(I)$ when we need to 
emphasize the defining monomial ideal $I$. 
\end{Definition}

The module $M_k$ clearly satisfies $\Fitt_0(F/M_k)=I_2(\tilde{M_k}) \supset I$, and, hence,
\begin{itemize}
\item the quotient $F/M_k$ has finite length, 
\item $\rank_R(M_k)=2$, and 
\item $I(M_k)=I_2(\tilde{M_k})$. 
\end{itemize}

Moreover, we have the following. 

\begin{Lemma}\label{3.3}
Let $1 \leq k < b_r$. Then the module $M_k$ satisfies the following. 
\begin{enumerate}
\item $\mu_R(M_k)=r+2$
\item $I(M_k)=I$, if $b_i+b_r-k \geq b_{i+1}$ for all $0 \leq i \leq r-1$
\end{enumerate}
\end{Lemma}

\begin{proof}
We show the assertion (1). Suppose that 
$$\sum_{i=0}^{r-1} \alpha_{i} \begin{pmatrix}
x^{a_{i}-1}y^{b_{i}} \\ 0 \end{pmatrix}
+\alpha_{r} \begin{pmatrix}
y^{k} \\ x \end{pmatrix}
+\alpha_{r+1} \begin{pmatrix}
0 \\ y^{b_{r}-k} \end{pmatrix}
=\begin{pmatrix} 
0 \\ 0 \end{pmatrix}
$$
where $\alpha_{i} \in R$ for $0 \leq i \leq r+1$. 
Then $\alpha_{r}x+\alpha_{r+1}y^{b_{r}-k}=0$. Thus,  
$\alpha_{r} \in (y^{b_{r}-k})$ and $\alpha_{r+1} \in (x)$.  
Write $\alpha_{r}=\beta y^{b_{r}-k}$ for some $\beta \in R$. 
Then $\alpha_{r}y^{k}=\beta y^{b_{r}}$ and  
$$\sum_{i=0}^{r-1} \alpha_{i}(x^{a_{i}-1}y^{b_{i}}) + \beta y^{b_{r}}=0. $$
Since $\alpha_{0} x^{a_{0}-1} \in (y^{b_{1}})$, $\alpha_{0} \in (y^{b_{1}})$. 
Let $1 \leq i \leq r-1$. Since
$\alpha_{i}(x^{a_{i}-1}y^{b_{i}}) \in (x^{a_{i-1}-1}, y^{b_{i+1}})$, 
it follows that 
$$\alpha_{i} \in (x^{a_{i-1}-a_{i}}, y^{b_{i+1}-b_{i}}).$$ 
Thus, $\alpha_{i} \in \fkm$ for all $0 \leq i \leq r+1$. This shows that $\mu_R(M_k)=r+2$. 

We show the assertion (2). It is clear that 
\begin{align*}
I(M_{k})&=I+I_{2}\begin{pmatrix}
x^{a_{0}-1} & \cdots &  x^{a_{i}-1}y^{b_{i}} & \cdots & x^{a_{r-1}-1}y^{b_{r-1}} & 0 \\
0 & \cdots & 0 & \cdots & 0 & y^{b_{r}-k} 
\end{pmatrix} \\
&=I+(x^{a_{i}-1}y^{b_{i}+b_{r}-k} \mid 0 \leq i \leq r-1). 
\end{align*}
For any $0 \leq i \leq r-1$, 
$$x^{a_{i}-1}y^{b_{i}+b_{r}-k}
=(x^{a_{i+1}}y^{b_{i+1}})x^{a_{i}-1-a_{i+1}}y^{b_{i}+b_{r}-k-b_{i+1}} \in I $$
because $a_{i}>a_{i+1}$ and $b_{i}+b_{r}-k \geq b_{i+1}$. 
Thus, $I(M_{k})=I$. \end{proof}

\begin{Remark}\label{3.4}
{\rm 
The condition in Lemma \ref{3.3} (2): 
$$b_i+b_r-k \geq b_{i+1} \ \text{for all} \ 0 \leq i \leq r-1$$
is satisfied, if either 
\begin{enumerate}
\item $1 \leq k \leq r-1$, or 
\item $r \leq k \leq b_{r-1}$ and $b_r-b_{r-1} \geq b_{i+1}-b_i$ for all $0 \leq i \leq r-1$. 
\end{enumerate}
}
\end{Remark}

\begin{proof}
Let $0 \leq i \leq r-1$. The case (1) follows from  
$$b_{i}+b_{r}-k-b_{i+1}=b_{r}-b_{i+1}+b_{i}-k \geq r-(i+1)+i-k=(r-1)-k \geq 0. $$ 
The case (2) follows from  
$$b_{i}+b_{r}-k-b_{i+1} \geq b_{i}+b_{r}-b_{r-1}-b_{i+1}
=(b_{r}-b_{r-1})-(b_{i+1}-b_{i}) \geq 0. $$ \end{proof}

\begin{Proposition}\label{3.5}
Let $1 \leq k < b_r$. Suppose that the monomial ideal $I$ is integrally closed and $I(M_k)=I$. 
Then the module $M_k$ is contracted from 
$S=R[ \frac{\fkm}{x+y} ]$. 
\end{Proposition}

\begin{proof}
Since $I$ is integrally closed, $I$ is contracted and $\mu_R(I)=\ord(I)+1$ by Propositions \ref{2.2} and \ref{2.3}. Thus, $\ord(I(M_k))=\ord(I)=r$ by Lemma \ref{3.1}. It follows that, by Lemma \ref{3.3}, 
$$\mu_R(M_k)=r+2=\ord(I(M_k))+\rank_R(M_k). $$ 
Note that $\ell_R(R/I(M_k)+(x+y))=\ord(I(M_k)) $ by Lemma \ref{3.1}.  
Thus, $M_k$ is contracted from $S=R[\frac{\fkm}{x+y}]$ by Proposition \ref{2.3}. \end{proof}

Here is the main result in this section, which plays an important role in this article. 

\begin{Theorem}\label{3.6}
Suppose that the monomial ideal $I$ is integrally closed. Then for any $1 \leq k \leq r-1$, the module $M_k$ is integrally closed with $I(M_k)=I$. 
\end{Theorem}

\begin{proof}
Let $1 \leq k \leq r-1$. By Remark \ref{3.4} (1), it follows that $I(M_k)=I$. 
Since $I$ is integrally closed and $I(M_{k})=I$, the module 
$M_k$ is contracted from $S=R[\frac{\fkm}{x+y}]$ 
by Proposition \ref{3.5}. To show that $M_k$ is integrally closed, it is enough to show that 
$M_kS$ is integrally closed by Proposition \ref{2.4}. 
This is equivalent to that $M_kS_Q$ is integrally closed for every maximal ideal $Q$ of $S$. 

Let $z:=\frac{x}{x+y} \in S$. Then we can write $x=z(x+y)$ and $y=(1-z)(x+y)$ in $S$. 
Thus, the matrix $\tilde{M_kS}$ over $S$ can be written as
\begin{align*}
\tilde{M_kS} &=\begin{pmatrix}
f_0 & \cdots & f_i & \cdots & f_{r-1} & (1-z)^k(x+y)^k & 0 \\
0 & \cdots & 0 & \cdots & 0 & z(x+y) & (1-z)^{b_r-k}(x+y)^{b_r-k}
\end{pmatrix}
\end{align*}
where 
$$f_i=z^{a_i-1}(1-z)^{b_i}(x+y)^{a_i+b_i-1} \ \text{for} \ 0 \leq i \leq r-1. $$
Here we note that 
\begin{itemize}
\item $a_i+b_i-1 \geq r-1 \geq k$ for all $0 \leq i \leq r-1$, 
\item $b_r-k \geq b_{r}-(r-1) \geq 1$.
\end{itemize} 
By considering an $S$-linear map 
$S^{2} \to S^{2}$ represented by a matrix 
$\begin{pmatrix}
(x+y)^{k} & 0 \\
0 & x+y
\end{pmatrix},$ 
we have that 
\begin{equation*}
\tilde{M_kS} \sim \begin{pmatrix}
g_{0} & \cdots & g_{i} & \cdots &  g_{r-1} & (1-z)^k & 0 \\
0 & \cdots & 0 & \cdots & 0 & z & (1-z)^{b_r-k}(x+y)^{b_r-k-1}
\end{pmatrix}
\end{equation*}
where 
$$g_{i}=z^{a_i-1}(1-z)^{b_i}(x+y)^{a_i+b_i-1-k} \ \text{for} \ 0 \leq i \leq r-1.$$ 

Let $Q$ be a maximal ideal of $S$. We show that $M_kS_Q$ is integrally closed. 
When $Q \nsupseteq \fkm S$. Then $S_Q$ is a discrete valuation ring. Thus, $M_kS_Q$ is integrally closed because of 
the fact that any submodule of finitely generated free module over a discrete valuation ring 
is integrally closed. 
Suppose that $Q \supseteq \fkm S$. 
When $z \notin Q$. Then $z$ is a unit of $S_Q$. By elementary matrix operations over $S_Q$, 
$$\tilde{M_kS_Q}\sim 
\begin{pmatrix}
h_{0} & \cdots & h_{i} & \cdots & h_{r-1} & 0 & (1-z)^{b_{r}}(x+y)^{b_r-k-1} \\
0 & \cdots & 0 & \cdots & 0 & 1 & 0
\end{pmatrix}$$
where 
$$h_{i}=(1-z)^{b_i}(x+y)^{a_i+b_i-1-k} \ \text{for} \ 0 \leq i \leq r-1.$$ 
This implies that $M_kS_Q \cong J \oplus S_Q$ for some $QS_Q$-primary ideal $J$ of $S_Q$. 
We then claim that $J$ is integrally closed. 
By Proposition \ref{2.5}, 
$$J=I(M_kS_Q)=I(M_k)^{S_Q}=I^{S_Q}. $$
Since $I$ is integrally closed, its transform $J$ is also integrally closed. 
Thus, $M_kS_Q$ is integrally closed. 
When $z \in Q$. Then $1-z \notin Q$ and it is a unit of $S_Q$. 
By elementary matrix operations over $S_Q$, 
$$\tilde{M_kS_Q}\sim 
\begin{pmatrix}
0 & \cdots & 0 & \cdots & 0 & 1 & 0 \\
h'_{0} & \cdots & h'_{i} & \cdots & h'_{r-1} & 0 & (x+y)^{b_r-k-1} \\
\end{pmatrix} $$
where 
$$h'_{i}=z^{a_i}(x+y)^{a_i+b_i-1-k} \ \text{for} \ 0 \leq i \leq r-1.$$ 
Thus, $M_kS_Q \cong S_Q \oplus J'$ for some $QS_Q$-primary ideal $J'$ of $S_Q$. 
Similarly, it follows that $J'$ is integrally closed. 
Thus, $M_kS_Q$ is integrally closed. 
This completes the proof. \end{proof}

\begin{Remark}
{\rm 
Let $M$ be a finitely generated torsion-free $R$-module, and 
let $\Rees(M)$ be the Rees algebra of $M$ which coincides with the subring $\Im(\Sym_R(M) \to \Sym_R(M^{**}))$
of a polynomial ring $\Sym_R(M^{**})$ over $R$. Suppose that $M$ is integrally closed. Then 
$\Rees(M)$ is a Noetherian normal domain by \cite[Theorem 5.3]{Ko}. Moreover, by \cite[Theorem 4.1]{KK}, 
it is Cohen-Macaulay. Therefore, by Theorem \ref{3.6}, we have a large class of concrete Cohen-Macaulay normal 
Rees algebras of modules. 
}
\end{Remark}

\section{Indecomposability}

In this section, we investigate the indecomposability of the modules introduced in section 3. 
So, we will work under the same situation and notations in section 3. 
Thus, $I$ is the monomial ideal considered in (\ref{mono}) and 
$M_{k}$ is the associated module introduced in Definition \ref{3.2}. 
The goal of this section is to show that if $\ord (I) \geq 3$, then
we can find $k$ such that $M_{k}$ is indecomposable integrally closed with $I(M_{k})=I$. 
Hereafter, throughout this section, we further assume that the ideal $I$ satisfies 
the following additional condition: 
\begin{equation}\label{*}
\begin{cases}
I \ \text{is integrally closed}, & \\
r=\ord(I) \geq 3, \ \text{and} & \\
a_0 \leq b_r. & \\
\end{cases}
\end{equation}

For the purpose, we first recall some known facts about the integral closure of general monomial ideals 
(not necessarily in a polynomial ring over a field) and its Zariski decomposition. 
We refer the readers to \cite{HbSw, HuSw, KiSt} for more results and the details on general monomial ideals. 

Let $\fka$ be an $\fkm$-primary monomial ideal in $R$ with respect to 
a regular system of parameters $x, y$. Suppose that $\fka$ is generated by a set of 
monomials $\{ x^{v_{i}}y^{w_{i}} \mid 1 \leq i \leq s \}$. Then, as in the usual monomial ideal case, 
one can define the Newton polyhedron ${\rm NP}(\fka)$ of $\fka$ as a convex hull of a set of exponent vectors of 
$\fka$ in $\mathbb R^2$. Namely, 
$${\rm NP}(\fka)=\left\{ (u_{1}, u_{2}) \mid (u_{1}, u_{2}) \geq 
\sum_{i=1}^{s} c_{i}(v_{i}, w_{i}) \ \text{for some} \ c_i \geq 0 \ \text{with} \ \sum_{i=1}^{s} c_{i}=1 \right\}. $$
Then the integral closure $\bar{\fka}$ of $\fka$ can be described as 
$$\bar{\fka}=( x^{e_{1}}y^{e_{2}} \mid (e_{1}, e_{2}) \in \mathbb Z_{\geq 0}^2 \cap {\rm NP}(\fka) ). $$
Thus, $\bar \fka$ is again a monomial ideal with respect to $x, y$. 

Let $\{(p_{i}, q_{i}) \mid 1 \leq i \leq t \}$ be a set of the vertices of ${\rm NP}(\fka)$ with 
$p_{0}>p_{1}>\dots >p_{t}=0$ and $q_{0}=0<q_{1}<\dots <q_{t}$. 
Then, by the above description of $\bar{\fka}$, it follows that 
$$\bar{\fka}=\bar{(x^{p_{i}}y^{q_{i}} \mid 1 \leq i \leq t)}. $$
Moreover, one can see that 
$$\bar{\fka}=\prod_{i=1}^t \bar{(x^{p_{i-1}-p_{i}}, y^{q_{i}-q_{i-1}})}. $$
Here we note that for a pair of positive integers $p', q'$ with ${\rm gcd}(p', q')=d$, 
$$\bar{(x^{p'}, y^{q'})}=\bar{(x^{p}, y^{q})}^{d}$$
where $p'=dp$ and $q'=dq$, and that for any $p, q>0$ with ${\rm gcd}(p, q)=1$, 
$$\bar{(x^p, y^q)} \ \text{is simple}.  $$
See \cite{GrKi} for more details on the above special simple ideals. 

Namely, for any $\fkm$-primary complete monomial ideal $\fka$ in $R$, 
every simple factor in the Zariski decomposition 
of $\bar{\fka}$ is a monomial ideal with the following special form: 
$$\bar{(x^{p}, y^{q})} \ \text{where} \ {\rm gcd}(p, q)=1. $$
We will illustrate these decompositions in Examples \ref{5.2} and \ref{5.3}. 
See \cite{BiTr, Qui} for more detailed and related results on the decomposition of usual monomial ideals.

Now, we begin with the following observation. This will be often used in our arguments. 

\begin{Observation}\label{4.1}
{\rm Let $1 \leq k < b_{r}$. We first note that the ideal $I$ satisfies that 
\begin{itemize}
\item $a_{r-1}=1$. 
\end{itemize}
This follows from the additional assumption that $I$ is integrally closed 
and $a_0 \leq b_r$. 
Thus, the ideal $I$ is of the form 
$$I=(x^{a_0}, x^{a_1}y^{b_1}, \dots , xy^{b_{r-1}}, y^{b_r}), $$
and the associated module $M_k$ is
$$M_k =\left\langle
\begin{pmatrix}
x^{a_0-1} & \cdots &  x^{a_i-1}y^{b_i} & \cdots & y^{b_{r-1}} & y^k & 0 \\
0 & \cdots & 0 & \cdots & 0 & x & y^{b_r-k}
\end{pmatrix} \right\rangle. $$
It follows that the other Fitting ideal is clearly of the form 
$$I_1(M_k)=(x, y^{\ell}) \ \text{where} \ \ell=\min\{b_{r-1}, k, b_r-k\}.$$ 
Here we assume that 
$M_{k}$ is integrally closed with $I(M_k)=I$, and $M_k$ is decomposable.
Then 
$$M_k \cong J_1 \oplus J_2$$
for some $\fkm$-primary ideals $J_1, J_2$ in $R$. 
Note that both $J_1$ and $J_2$ are integrally closed ideals in $R$ because 
$M_{k}$ is assumed to be integrally closed 
and $\bar{J_{1}\oplus J_{2}}=\bar{J_{1}} \oplus \bar{J_{2}}$. 
Consider the associated Fitting ideals of $M_{k}$. Then we have equalities
\begin{align*}
J_1J_2 &=I(M_k)=I, \\
J_1+J_2 &= I_1(M_k)=(x, y^{\ell}). 
\end{align*}
The first equality implies that both $J_1$ and $J_2$ are a part of factors in the 
Zariski decomposition 
of $I$. This implies that both $J_{1}$ and $J_{2}$ are monomial ideals. 
Thus, the sum $J_1+J_2$ is also a monomial ideal. Therefore, 
as in the usual monomial ideal case (see \cite[Corollary 3]{KiSt} for instance), 
the second equality implies that 
\begin{itemize}
\item $x \in J_1$ or $x \in J_2$, and 
\item $y^{\ell} \in J_1$ or $y^{\ell} \in J_2$.
\end{itemize} 
We may assume that $x \in J_1$. Thus, $\ord(J_1)=1. $ 
If $y^{\ell} \in J_2$, then $xy^{\ell} \in J_1J_2=I$. 
Therefore, if $xy^{\ell} \notin I$, then $y^{\ell} \in J_1$ so that $J_1=(x, y^{\ell})$ because
$(x, y^{\ell}) \subset J_1 \subset J_1+J_2 =(x, y^{\ell}).$
Consequently, we can summarize the observation as follows: 
}

\

If the module $M_{k}$ is integrally closed with $I(M_k)=I$, and $M_k$ is decomposable, then 
\begin{enumerate}
\item the monomial ideal $I$ has a simple factor of order $1$ in the Zariski decomposition. 
\item Moreover, if $xy^{\ell} \notin I$ where $\ell=\min\{b_{r-1}, k, b_r-k\}$, 
then $I$ has a simple factor of the form 
$(x, y^{\ell})$. 
\end{enumerate}
\end{Observation}

By Observation \ref{4.1}, we can readily get the following. 

\begin{Theorem}\label{4.2}
Suppose that the monomial ideal $I$ has no simple factor of order $1$ in the Zariski decomposition.
Then for any $1 \leq k \leq r-1$, 
the associated module $M_k$ is indecomposable integrally closed with $I(M_k)=I$. 
\end{Theorem}

\begin{proof}
Let $1 \leq k \leq r-1$. Since $I$ is integrally closed, $M_{k}$ is integrally closed with $I(M_{k})=I$ by Theorem \ref{3.6}. 
Suppose that $M_{k}$ is decomposable. By Observation \ref{4.1}, the ideal $I$ has a simple factor of 
order $1$ in the Zariski decomposition. This is a contradiction. 
\end{proof}

We next consider the case that the ideal $I$ has a simple factor of order $1$ 
in the Zariski decomposition. 
We then divide the case into the following two cases. Here we write $J \mid I$ if $I=J\fkb$ for some ideal 
$\fkb$ in $R$. 
$$\begin{cases}
\text{Case I} & (x, y^{\ell}) \nmid I \ \text{for some} \ 1 \leq \ell \leq r-1 \\
\text{Case II} & (x, y^{\ell}) \mid I \ \text{for any} \ 1 \leq \ell \leq r-1 \\
\end{cases}
$$

We begin with Case I. 

\begin{Observation}\label{4.3}
{\rm 
Suppose Case I and let $k_0=\min\{\ell \mid (x, y^{\ell}) \nmid I \}$. 
Since $1 \leq k_{0} \leq r-1$, $M_{k_{0}}$ is integrally closed with $I(M_{k_{0}})=I$ by Theorem \ref{3.6}. 
We consider the following condition: 
\begin{equation}\label{**}
I_1(M_{k_0})=(x, y^{k_0}) \ \text{and} \ xy^{k_0} \notin I. 
\end{equation}
If the condition (\ref{**}) is satisfied, and $M_{k_0}$ is decomposable, then 
$(x, y^{k_0}) \mid I$ by Observation \ref{4.1}. This is a contradiction. 
Thus, we have the following:}

\medskip

The condition $(\ref{**})$ implies that $M_{k_0}$ is indecomposable. 

\medskip

\noindent
{\rm 
Moreover, by elementary calculations as we will see in Proposition \ref{4.4}, one can see that }

\medskip

the ideal $I$ which does not satisfy the condition $(\ref{**})$ is any one of the following cases$:$ 
$$
\begin{cases}
(N_1) & I=(x, y)(x^{\alpha}, y)(x, y^{\beta}) \ \text{where} \ \beta \geq \alpha >0\\
(N_2) & I=(x, y)^3(x, y^2) \\
(N_3) & I=(x, y)^2(x, y^2)(x^2, y) \\
(N_4) & I=(x, y)(x, y^2)\bar{(x^3, y^2)} \\
\end{cases}
$$

\medskip

\noindent
{\rm 
Consequently, when Case I, we may only consider the above $4$ cases. 
}
\end{Observation}

\begin{Proposition}\label{4.4}
Suppose that the Zariski decomposition of $I$ is of the form
$$I=(x, y)(x, y^2) \cdots (x, y^{k-1})\fkb$$
for some $1 \leq k \leq r-1$. Consider the following condition$:$  
\begin{align*}
(F_k): \ \ I_1(M_k)=(x, y^k) \ \text{and}\ xy^k \notin I. 
\end{align*}
Then we have the following. 
\begin{enumerate}
\item If either $k \leq r-2$ or $r \geq 5$, then $(F_k)$ is satisfied. 
\item When $r=3$ and $k=2$. Then $(F_2)$ is satisfied except for the case $(N_1)$.
\item When $r=4$ and $k=3$. Then $(F_3)$ is satisfied except for the cases $(N_2), (N_3), (N_4)$.
\end{enumerate}
\end{Proposition}

\begin{proof}
Let $I=(x, y)(x,y^{2}) \cdots (x, y^{k-1})\fkb$. 
The ideal $\fkb$ is a factor in the Zariski decomposition of $I$ so that  
it is an integrally closed monomial ideal with respect to $x, y$ of $\ord(\fkb)=r-(k-1)$. 
Thus,
$$I \subset (x, y^{1+2+\dots +(k-1)+r-(k-1)})=(x, y^{\frac{(k-1)(k-2)}{2}+r}). $$
Since $y^{b_r} \in I$, we have that 
$$b_r \geq \frac{(k-1)(k-2)}{2}+r. $$

We first show the assertion (1). 
Suppose that either $k \leq r-2$ or $r \geq 5$. Then it is easy to see that $b_r-k \geq k$ and hence 
$I_1(M_k)=(x, y^k)$. 
When $k \leq r-2$. The assertion $xy^k \notin I$ is clear 
because $\ord(I)=r$. When $r \geq 5$ and $k=r-1$. 
Then 
$$I \subset (x^2, xy^{1+2+\dots +(r-3)+2}, y^{1+2+\dots +(r-2)+2}). $$
Since $xy^{b_{r-1}} \in I$ and $r \geq 5$, it follows that 
$$b_{r-1} \geq 1+2+\dots +(r-3)+2=\frac{(r-2)(r-3)}{2}+2>r-1. $$
This implies that $xy^{r-1} \notin I$. We have the assertion (1). 

We next show the assertion (2). 
Suppose that $r=3$ and $k=2$. Then $I=(x,y)\fkb$ and $\ord(\fkb)=2$. Thus, we can write 
$$\fkb=(x^a, x^{a'}y^{b'}, y^b)$$ 
where $a>a'>0$, $b>b'>0$ and $a \leq b$. 
If $a'=b'=1$, then $xy^2 \in I$, and, hence, $(F_2)$ is not satisfied. 
When this is the case, 
$$I=(x,y)(x^{a-1}, y)(x,y^{b-1}) $$
which is the ideal in case $(N_{1})$. 
Suppose that $(a', b') \neq (1, 1)$. Then $a=2$ and $a'=1$ because $\ord(\fkb)=2$. Thus,
$$I=(x,y)(x^2, xy^{b'}, y^b)=(x^3, x^2y, xy^{b'+1}, y^{b+1}). $$
Note that $b>b' \geq 2$. Thus, $xy^2 \notin I$. Since $b_{3}-2=(b+1)-2=b-1 \geq 2$, $I_1(M_2)=(x, y^2)$. 
It follows that 
$(F_2)$ is satisfied when $(a', b') \neq (1, 1)$. We have the assertion (2). 

Finally, we show the assertion (3). Suppose that $r=4$ and $k=3$. Then $I=(x,y)(x,y^{2})\fkb$ and 
$\ord(\fkb)=2$. Thus, we can write 
$$\fkb=(x^a, x^{a'}y^{b'}, y^b)$$ 
where $a>a'>0$, $b>b'>0$ and $a+2 \leq b+3$. Thus, $2 \leq a \leq b+1$. 
If $b=2$, then $xy^3 \in I$, and, hence, $(F_3)$ is not satisfied. 
When this is the case, $a=2$ or $a=3$. Thus, 
$$
\fkb=\begin{cases}
(x^{2}, xy, y^{2})=(x, y)^2,  & or \\
(x^{3}, xy, y^{2})=(x^2, y)(x,y),  & or \\
(x^{3}, x^{2}y, y^{2})=\bar{(x^3, y^2)}.  & 
\end{cases}
$$
These are cases in $(N_2), (N_3), (N_4)$. 
Suppose that $b \geq 3$. Then the assertion $xy^3 \notin I$ is clear. Since $b_4-3=(b+3)-3=b \geq 3$, $I_1(M_3)=(x, y^3)$. 
Thus, $(F_3)$ is satisfied when $b \geq 3$. We have the assertion (3). \end{proof}

The ideal in cases $(N_1), (N_2), (N_3)$ can be regarded as a special case of the form
$$I=(x,y)^{r-2}(x^{\alpha}, y)(x,y^{\beta})$$
where $r \geq 3$ and $\alpha, \beta \geq 1$. In this case, one can see that 
the associated module $M_{r-2}$ is indecomposable. 

\begin{Proposition}\label{4.5}
Let $I=(x,y)^{r-2}(x^{\alpha}, y)(x,y^{\beta})$ where $r \geq 3$ and $\alpha, \beta \geq 1$. Then 
$$M_{r-2}=
\left\langle
\begin{pmatrix}
x^{\alpha+r-2} & x^{r-2}y & \dots & x^{r-1-i}y^i &  \cdots &  y^{r-1} & y^{r-2} & 0 \\
0 & 0 & \cdots & 0& \cdots & 0 & x & y^{\beta+1} 
\end{pmatrix}
\right\rangle $$ 
is indecomposable integrally closed with $I(M_{r-2})=I$. 
\end{Proposition}

\begin{proof}
Since $I$ is integrally closed, 
$M_{r-2}$ is integrally closed with $I(M_{r-2})=I$ by Theorem \ref{3.6}. 
We show the indecomposability. Note that 
\begin{align*}
I_1(M_{r-2}) &= \begin{cases}
(x, y^{\beta+1}) & (\beta<r-3) \\
(x, y^{r-2}) & (\beta \geq r-3). \\
\end{cases}
\end{align*}

When $\beta< r-3$. It is clear that $xy^{\beta+1} \notin I$. Assume that $M_{r-2}$ is decomposable. Then 
$(x, y^{\beta+1}) \mid I$ by Observation \ref{4.1}. This is a contradiction. 
Thus, $M_{r-2}$ is indecomposable when $\beta < r-3$. 

When $\beta \geq r-3$. It is clear that $xy^{r-2} \notin I$. Assume that $M_{r-2}$ is decomposable. Then 
$(x, y^{r-2}) \mid I$ by Observation \ref{4.1}. Moreover, Observation \ref{4.1} tell us that 
\begin{equation*}
M_{r-2} \cong \begin{cases}
(x, y) \oplus (x^{\alpha},y)(x,y^{\beta}) & (r=3) \\
(x,y^{r-2}) \oplus (x,y)^{r-2}(x^{\alpha},y) & (r \geq 4)
\end{cases}
\end{equation*}
Here we note that $\beta =r-2$ when $r \geq 4$. Thus, 
\begin{equation*}
\ell_R(F/M_{r-2})=\begin{cases}
\alpha+\beta+2 & (r=3) \\
\frac{(r-2)(r+3)}{2}+\alpha & (r \geq 4)
\end{cases}
\end{equation*}
On the other hand, since
\begin{equation*}
M_{r-2} \subset 
\left\langle
\begin{pmatrix}
x^{\alpha+r-2} & x^{r-2}y & \cdots & x^{r-1-i}y^i & \cdots & x^2y^{r-3} & y^{r-2} & 0 & 0 \\
0 & 0 & \cdots & 0 & \cdots & 0 & 0 & x & y^{\beta+1} 
\end{pmatrix}
\right\rangle,
\end{equation*}
we have a surjective $R$-linear map 
\begin{equation*}
\eta: F/M_{r-2} \to R/(x^{\alpha+r-2}, x^{r-2}y, \dots , x^2y^{r-3}, y^{r-2}) \oplus R/(x, y^{\beta+1}). 
\end{equation*}
This implies that 
\begin{equation*}
\ell_R(F/M_{r-2}) \geq \begin{cases}
\alpha+\beta+2 & (r=3) \\
\frac{(r-2)(r+3)}{2}+\alpha & (r \geq 4)
\end{cases}
\end{equation*}
Hence, $\eta$ is an isomorphism. Considering the first Fitting ideal, we have equalities
\begin{equation*}
I=I(M_{r-2})=(x^{\alpha+r-2}, x^{r-2}y, \dots , x^2y^{r-3}, y^{r-2})(x, y^{\beta+1})
\end{equation*}
which is a contradiction. This proves that $M_{r-2}$ is indecomposable. \end{proof}

The remaining case in Case I is the ideal of type $(N_4)$. 

\begin{Example}\label{4.6}
Let $I=(x,y)(x,y^2)\bar{(x^3,y^2)}$. Then 
\begin{equation*}
M_2=\left\langle 
\begin{pmatrix}
x^4 & x^3y & xy^2 & y^3 & y^2 & 0 \\
0 & 0& 0 & 0 & x & y^3
\end{pmatrix}
\right\rangle
\end{equation*}
is indecomposable integrally closed with $I(M_2)=I$. 
\end{Example}

\begin{proof}
By Theorem \ref{3.6}, $M_2$ is integrally closed with $I(M_2)=I$. We need to show the indecomposability. 
It is clear that $I_1(M_2)=(x, y^2)$ and $xy^2 \notin I$. If $M_2$ is decomposable, then 
\begin{equation*}
M_2 \cong (x,y^2) \oplus (x,y)\bar{(x^3, y^2)}
\end{equation*}
by Observation \ref{4.1}. 
Hence, $\ell_R(F/M_2)=2+8=10$. On the other hand, since 
\begin{equation*}
M_2 \subset 
\left\langle
\begin{pmatrix}
x^4 &  x^3y & y^2 & 0 & 0  \\
0 & 0 & 0 & x & y^3 
\end{pmatrix}
\right\rangle, 
\end{equation*}
we have a surjective $R$-linear map 
\begin{equation*}
\eta: F/M_2 \to R/(x^4, x^3y,y^2) \oplus R/(x,y^3). 
\end{equation*}
Thus, $\ell_R(F/M_2) \geq 7+3=10$, and, hence, $\eta$ is an isomorphism. This implies equalities
\begin{equation*}
I=I(M_2)=(x^4, x^3y,y^2)(x,y^3)
\end{equation*}
which is a contradiction. This shows that $M_2$ is indecomposable. \end{proof}

As a consequence, we get the following result in Case I. 

\begin{Theorem}\label{4.7}
Suppose that $(x,y^{\ell}) \nmid I$ for some $1 \leq \ell \leq r-1$. 
Then there exists an integer $1 \leq k \leq r-1$ such that 
$M_k$ is indecomposable integrally closed with $I(M_k)=I$. 
\end{Theorem}

We move to Case II. The ideal $I$ is of the form
\begin{equation*}
I=(x,y)(x,y^2) \cdots (x,y^{r-1})\fkb
\end{equation*}
where $\fkb$ is a simple factor of the monomial ideal $I$ with $\ord(\fkb)=1$. Thus,   
$$\fkb=(x^{\alpha}, y) \ \text{or} \ \fkb=(x,y^{\beta})$$ 
where $\alpha, \beta \geq 1$. We divide Case II into the following two cases: 
\begin{equation*}
\begin{cases}
\text{Case II--1} & \fkb=(x^{\alpha}, y)  \ \text{where} \ \alpha \geq 1 \\
\text{Case II--2} & \fkb=(x, y^{\beta})  \ \text{where} \ \beta \geq 2 
\end{cases}
\end{equation*}

We first consider Case II--1. When $r=3$. The ideal $I=(x,y)(x,y^2)(x^{\alpha},y)$ can be viewed 
as a special case in Proposition \ref{4.5}. 
Thus, $M_1$ is indecomposable in this case. When $r=4$. One can see that $M_3$ is indecomposable
as follows. 

\begin{Example}\label{4.8}
Let $I=(x,y)(x,y^2)(x,y^3)(x^{\alpha},y)$ where $\alpha \geq 1$. Then 
\begin{equation*}
M_3=\left\langle \begin{pmatrix}
x^{\alpha+2} & x^2y & xy^2 & y^4 & y^3 & 0 \\
0 & 0 & 0 & 0 & x & y^4
\end{pmatrix}
\right\rangle
\end{equation*}
is indecomposable integrally closed with $I(M_3)=I$. 
\end{Example}

\begin{proof}
We need to show the indecomposability of $M_3$. 
It is clear that $I_1(M_3)=(x,y^3)$ and $xy^3 \notin I$. 
If $M_3$ is decomposable, then 
$$M_3 \cong (x,y^3) \oplus (x,y)(x,y^2)(x^{\alpha},y) $$ 
by Observation \ref{4.1}. Thus, 
$\ell_R(F/M_3)=3+(\alpha+6)=\alpha+9$. On the other hand, since
\begin{equation*}
M_3 \subset 
\left\langle 
\begin{pmatrix}
x^{\alpha+2} & x^2y & xy^2 & y^3 & 0 & 0 \\
0 & 0 & 0 & 0 & x & y^4 
\end{pmatrix}
\right\rangle, 
\end{equation*}
we have a surjective $R$-linear map 
\begin{equation*}
\eta: F/M_3 \to R/(x^{\alpha+2}, x^2y, xy^2,y^3) \oplus R/(x,y^4). 
\end{equation*}
By comparing length, $\eta$ is an isomorphism. This implies equalities 
\begin{equation*}
I=I(M_3)=(x^{\alpha+2}, x^2y, xy^2,y^3)(x,y^4)
\end{equation*}
which is a contradiction. This shows that $M_3$ is indecomposable. \end{proof}

When $r=5$, one can see that $M_{4}$ is indecomposable in the same manner. However, 
when $r \geq 6$, the same approach as in Example \ref{4.8} does not work, and 
it seems to be difficult to find an indecomposable module $M_k$ in the range 
$1 \leq k \leq r-1$. Therefore, we consider the next module $M_r$.

\begin{Theorem}\label{4.9}
Let $I=(x,y)(x,y^2) \cdots (x,y^{r-1})(x^{\alpha},y)$ where $r \geq 5$ and $\alpha \geq 1$. Then 
$M_r$ is indecomposable integrally closed with $I(M_r)=I$. 
\end{Theorem}

\begin{proof}
Note that $I=(x^{\alpha+r-1})+(x^{r-i}y^{b_i} \mid 1 \leq i \leq r)$ where 
$$b_i=1+1+2+\dots +(i-1)=\frac{i(i-1)}{2}+1. $$ 
Then one can easily see that the ideal $I$ satisfies the condition in Remark \ref{3.4} (2): 
\begin{equation*}
b_{r-1} \geq r \ \text{and} \ b_r-b_{r-1} \geq b_{i+1}-b_i \ \text{for all} \ 0 \leq i \leq r-1. 
\end{equation*}
Thus, $I(M_r)=I$ by Lemma \ref{3.3}. Since $I$ is integrally closed with $I(M_{r})=I$, 
the module $M_r$ is contracted from $S=R[\frac{\fkm}{x+y}]$ 
by Proposition \ref{3.5}. 
To show that $M_r$ is integrally closed, it is enough to show that $M_rS$ is integrally closed 
by Proposition \ref{2.4}. This is equivalent to that $M_rS_Q$ is integrally closed 
for every maximal ideal $Q$ of $S$. 

Let $z:=\frac{x}{x+y} \in S$. Then we can write $x=z(x+y)$ and $y=(1-z)(x+y)$ in $S$. 
As in the proof of Theorem \ref{3.6}, by considering an $S$-linear map $S^{2} \to S^{2}$ represented by 
$\begin{pmatrix}
(x+y)^{r-1} & 0 \\
0 & x+y
\end{pmatrix}
$, we have that 
\begin{equation*}
\tilde{M_rS} \sim 
\begin{pmatrix}
z^{\alpha+r-2}(x+y)^{\alpha-1} & f_{1} & \cdots & f_{i} & \cdots & f_{r-1} & (1-z)^r(x+y) & 0 \\
0 & 0 & \cdots &  0 & \cdots & 0 & z & f
\end{pmatrix}
\end{equation*}
where 
$$
\begin{cases}
f_{i}=z^{r-i-1}(1-z)^{\frac{i(i-1)}{2}+1}(x+y)^{\frac{(i-1)(i-2)}{2}} &  \text{for} \ 
1 \leq i \leq r-1 \\
f= (1-z)^{\frac{(r-1)(r-2)}{2}}(x+y)^{\frac{(r-1)(r-2)}{2}-1}  &
\end{cases}  
$$

Let $Q$ be a maximal ideal of $S$. We show that $M_rS_Q$ is integrally closed. 
When $Q \nsupseteq \fkm S$. Since $S_Q$ is a discrete valuation ring, 
it follows that $M_{r}S_{Q}$ is integrally closed. 
Suppose that $Q \supseteq \fkm S$. When $z \notin Q$. 
Then $z$ is a unit of $S_Q$. By elementary matrix operations over $S_Q$,   
\begin{equation*}
\tilde{M_rS_Q} \sim \begin{pmatrix}
(x+y)^{\alpha-1} & 1-z & 0 \\
0 & 0 & 1
\end{pmatrix}. 
\end{equation*}
This implies that 
$M_rS_Q \cong J \oplus S_Q$ for some $QS_Q$-primary integrally closed ideal $J$ of $S_Q$. 
Therefore, $M_rS_Q$ is integrally closed. 
Suppose that $z \in Q$. Then $1-z \notin Q$ and it is a unit of $S_Q$. 
By elementary matrix operations over $S_Q$,  
\begin{equation}\label{***}
\tilde{M_rS_Q} \sim \begin{pmatrix}
z^{r-3} & g_{3} & \cdots & g_{i} & \cdots & g_{r-1} & x+y & 0 \\
0 & 0 & \cdots & 0 & \cdots & 0 & z & g
\end{pmatrix}
\end{equation}
where 
$$
\begin{cases}
g_{i}= z^{r-i-1}(x+y)^{\frac{(i-1)(i-2)}{2}} & \text{for} \ 3 \leq i \leq r-1 \\
g=(x+y)^{\frac{(r-1)(r-2)}{2}-1}. & 
\end{cases}
$$ 
We consider the ideal in $S_Q$:  
\begin{equation*}
\fkc=(z,x+y)(z, (x+y)^2) \cdots (z, (x+y)^{r-2}). 
\end{equation*}
Note that $z, x+y$ is a regular system of parameters for $S_Q$, and 
$\fkc$ is an integrally closed monomial ideal with respect to $z, x+y$. 
We then claim the following. 

\

\noindent
{\bf Claim} $M_rS_Q \cong M_1(\fkc)$

\

Note that $\fkc=(z^{r-2-j}(x+y)^{c_j} \mid 0 \leq j \leq r-2)$, where 
$c_j=1+\dots +j=\frac{(j+1)j}{2}, $
and  
\begin{align*}
\tilde{M_1(\fkc)} & = 
\begin{pmatrix}
z^{r-3} & h_{1} & \cdots & h_{j} & \cdots & h_{r-3} & x+y & 0 \\
0 & 0 & \cdots & 0 & \cdots & 0 & z & h 
\end{pmatrix}
\end{align*}
where 
$$
\begin{cases}
h_{j}=z^{r-2-j-1}(x+y)^{c_j} & \text{for} \ 1 \leq j \leq r-3 \\
h=(x+y)^{c_{r-2}-1}. & 
\end{cases}
$$ 
Thus, it follows that $\tilde{M_1(\fkc)} \sim \tilde{M_rS_Q}$ by (\ref{***}) and hence 
$M_rS_Q \cong M_1(\fkc)$. Since $M_1(\fkc)$ is integrally closed by Theorem \ref{3.6}, 
$M_rS_Q$ is integrally closed. Thus, we have that $M_{r}$ is integrally closed. 

Finally, we show the indecomposability. 
Since $r \geq 5$, it follows that $b_r-r \geq r$ and $b_{r-1} > r$. 
Thus, $I_1(M_r)=(x,y^r)$ and $xy^r \notin I$. If $M_r$ is decomposable, 
then $(x, y^r) \mid I$ by Observation \ref{4.1}. This is a contradiction. 
This shows that $M_r$ is indecomposable. \end{proof}

We next consider Case II--2. 
When $r=3$, one can see that $M_2$ is indecomposable as follows.   

\begin{Example}\label{4.10}
Let $I=(x,y)(x,y^2)(x,y^{\beta})$ where $\beta \geq 2$. Then 
\begin{equation*}
M_2=\left\langle
\begin{pmatrix}
x^2 & xy & y^3 & y^2 & 0 \\
0 & 0 & 0 & x & y^{\beta+1}
\end{pmatrix}
\right\rangle
\end{equation*}
is indecomposable integrally closed with $I(M_2)=I$. 
\end{Example}

\begin{proof}
We need to show the indecomposability of $M_{2}$. 
It is clear that $I_1(M_2)=(x, y^2)$ and $xy^2 \notin I$. 
If $M_2$ is decomposable, then 
$$M_2 \cong (x,y^2) \oplus (x,y)(x,y^{\beta}) $$ 
by Observation \ref{4.1}. 
Hence, $\ell_R(F/M_2)=2+(\beta+2)=\beta+4$. 
On the other hand, since
\begin{equation*}
M_2 \subset 
\left\langle
\begin{pmatrix}
x^2 &  xy & y^2 & 0 & 0 \\
0 & 0 & 0 & x & y^{\beta+1} 
\end{pmatrix}
\right\rangle, 
\end{equation*}
we have a surjective $R$-linear map 
\begin{equation*}
\eta: F/M_2 \to R/(x,y)^2 \oplus R/(x,y^{\beta+1}). 
\end{equation*}
By comparing length, $\eta$ is an isomorphism. This implies equalities 
\begin{equation*}
I=I(M_2)=(x,y)^2(x,y^{\beta+1})
\end{equation*}
which is a contradiction. This shows that $M_2$ is indecomposable. \end{proof}

When $r=4$, one can see that $M_{3}$ is indecomposable in the same manner. However, 
when $r \geq 5$, the same approach as in Example \ref{4.10} 
does not work at least when $\beta \gg 0$, and 
it seems to be difficult to find indecomposable modules $M_k$ in the range $1 \leq k \leq r-1$.  
Therefore, we consider the next modules $M_r$ and $M_{r+1}$. 

\begin{Theorem}\label{4.11}
Let $I=(x,y)(x,y^2) \cdots (x,y^{r-1})(x,y^{\beta})$ where $r \geq 4$ and $\beta \geq 2$. 
Then $M_{r}$ and $M_{r+1}$ are integrally closed with $I(M_r)=I(M_{r+1})=I$. Moreover, 
$M_r$ is indecomposable if $\beta \neq r$, and $M_{r+1}$ is indecomposable if $\beta=r$. 
\end{Theorem}

\begin{proof}
Note that $I=(x^{r-i}y^{b_i} \mid 0 \leq i \leq r)$ where 
$$
\begin{cases}
b_i = \begin{cases}
\frac{i(i-1)}{2}+\beta & (\beta < i) \\
\frac{(i+1)i}{2} & (\beta \geq i)
\end{cases} &  \text{for} \ 0 \leq i \leq r-1 \\
b_r=\frac{r(r-1)}{2}+\beta. & 
\end{cases}
$$
Then one can easily see that the ideal $I$ satisfies the condition in Remark \ref{3.4} (2): 
\begin{equation*}
b_{r-1} \geq r+1 \geq r \ \text{and} \ b_r-b_{r-1} \geq b_{i+1}-b_i \ 
\text{for all} \ 0 \leq i \leq r-1.
\end{equation*}
Thus, $I(M_r)=I(M_{r+1})=I$ by Lemma \ref{3.3}. 
One can also easily see that 
$$I+(y)=(x^r, y) \ \text{and} \ \ell_R(R/I+(y))=r=\ord(I). $$ 
It follows that $M_r$ and $M_{r+1}$ are contracted from $S=R[\frac{\fkm}{y}]$ 
by Proposition \ref{2.3}. 

In what follows, let $k_0 \in \{r, r+1\}$. 
To show that $M_{k_0}$ is integrally closed, it is enough to show that $M_{k_0}S_Q$ is integrally closed 
for every maximal ideal $Q$ of $S$ with $Q \supseteq \fkm S$. 

Let $z :=\frac{x}{y} \in S$. Then $x=zy$ in $S$. 
As in the proof of Theorem \ref{3.6}, by considering an $S$-linear map $S^{2} \to S^{2}$ represented by 
$\begin{pmatrix}
y^{r-1} & 0 \\
0 & y
\end{pmatrix}
$, we have that 
\begin{equation}\label{****}
\tilde{M_{k_0}S} \sim 
\begin{pmatrix}
z^{r-2} & z^{r-3}y & f_{3} & \cdots & f_{i} & \cdots & f_{r-1} & y^{k_0-r+1} & 0 \\
0 & 0 & 0 & \cdots & 0 & \cdots & 0 & z & y^{b_r-k_0-1}
\end{pmatrix}
\end{equation}
where 
$$f_{i}=z^{r-i-1}y^{b_i-i} \ \ \text{for} \ 3 \leq i \leq r-1. $$

Let $Q$ be a maximal ideal of $S$ with $Q \supseteq \fkm S$. We show that 
$M_rS_Q$ is integrally closed. 
When $z \notin Q$. Then $z$ is a unit of $S_Q$. By elementary matrix operations over $S_Q$, 
\begin{equation*}
\tilde{M_{k_0}S_Q} \sim \begin{pmatrix}
1 & 0 \\
0 & 1 
\end{pmatrix}. 
\end{equation*}
Thus, $M_{k_0}S_Q \cong S_Q \oplus S_Q$ so that $M_{k_0}S_Q$ is integrally closed. 
Suppose that $z \in Q$. We consider the ideal in $S_Q$:  
\begin{equation*}
\fkc=(z, y)(z, y^2) \cdots (z, y^{r-2})(z,y^{\beta-1}). 
\end{equation*}
Note that $z, y$ is a regular system of parameters for $S_{Q}$, and $\fkc$ is an integrally closed monomial 
ideal with respect to $z, y$. 
We then claim the following. 

\

\noindent
{\bf Claim} $M_{k_0}S_Q \cong M_{k_0-r+1}(\fkc)$

\

Note that $\fkc=(z^{r-1-j}y^{c_j} \mid 0 \leq j \leq r-1)$ 
where 
$$
\begin{cases}
c_j = \begin{cases}
\frac{j(j-1)}{2}+(\beta-1) & (\beta-1<j) \\
\frac{(j+1)j}{2} & (\beta-1 \geq j)
\end{cases} & \text{for} \ 0 \leq j \leq r-2 \\
c_{r-1}=\frac{(r-1)(r-2)}{2}+(\beta-1). & 
\end{cases}
$$
Thus,
\begin{equation*}
\tilde{M_{k_0-r+1}(\fkc)} = 
\begin{pmatrix}
z^{r-2} & z^{r-3}y & g_{2} & \cdots & g_{j} & \cdots & g_{r-2} & y^{k_0-r+1} & 0 \\
0 & 0 & 0 & \cdots & 0 & \cdots & 0 & z & y^{c_{r-1}-k_0+r-1}
\end{pmatrix}
\end{equation*}
where 
$$g_{j}=z^{r-1-j-1}y^{c_j} \ \ \text{for} \ 2 \leq j \leq r-2. $$ 
Then it is easy to see that 
\begin{equation*}
b_r-r=c_{r-1} \ \text{and} \ b_i-i =c_{i-1} \ \text{for all} \ 3 \leq i \leq r-1. 
\end{equation*}
Thus, $\tilde{M_{k_0-r+1}(\fkc)} \sim \tilde{M_{k_0}S_Q}$ by (\ref{****}) and hence
$M_{k_0}S_Q \cong M_{k_0-r+1}(\fkc)$. Since $k_0-r+1 \in \{1, 2\}$, it follows that 
$M_{k_0-r+1}(\fkc)$ is integrally closed by Theorem \ref{3.6}. 
Thus, $M_{k_0}S_Q$ is integrally closed, and, hence, $M_{k_0}$ is integrally closed. 

Finally, we show the last assertion. 
Note that $b_r-k_0 \geq k_0$ and $b_{r-1} \geq k_0$ because $r \geq 4$. 
Hence, $I_1(M_{k_0})=(x,y^{k_0})$. Note that $xy^{k_{0}} \notin I$ except for the case 
where $k_{0}=r+1, r=4$ and 
$\beta=2$. When $\beta \neq r$. Then $xy^{r} \notin I$. 
If $M_r$ is decomposable, then $(x,y^r) \mid I$ by Observation \ref{4.1}. 
This is a contradiction. When $\beta=r$. Then $xy^{r+1} \notin I$. 
If $M_{r+1}$ is decomposable, then 
$(x,y^{r+1}) \mid I$ by Observation \ref{4.1}. This is a contradiction. 
Therefore, we have the last assertion. \end{proof}

As a consequence, we get the following result in Case II. 

\begin{Theorem}\label{4.12}
Suppose that $(x, y^{\ell}) \mid I$ for any $1 \leq \ell \leq r-1$. Then there exists an integer $1 \leq k \leq r+1$ such that $M_k$ is indecomposable integrally closed with $I(M_k)=I$. 
\end{Theorem}

\section{Proof of Theorem \ref{main} and examples}

We are now ready to complete a proof of Theorem \ref{main}. 

\begin{proof}[Proof of Theorem \ref{main}]
Let $(R, \fkm)$ be a two-dimensional regular local ring with a regular system of parameters $x, y$. 
Let $I$ be an $\fkm$-primary integrally closed monomial ideal with respect to $x, y$ and let 
$r:=\ord(I) \geq 2$. Then one can write the ideal $I$ as the following form: 
$$I=(x^{a_0}, x^{a_1}y^{b_1}, \dots , x^{a_{r-1}}y^{b_{r-1}}, y^{b_r})$$
where $a_{0} > a_{1} > \dots > a_{r-1} > a_{r}=0$ and $b_{0}=0 < b_{1} < \dots < b_{r}$. 
Without loss of generality, we may assume that $a_{0} \leq b_{r}$. 
Then, since $I$ is integrally closed, one can easily see that $a_{r-1}=1$. 

The case where $r \geq 3$ is done in section 3. Indeed, if the ideal $I$ has no simple factor of 
order $1$ in the Zariski decomposition, then for any $1 \leq k \leq r-1$, 
the associated module $M_{k}$ is indecomposable integrally closed with $I(M_{k})=I$ by Theorem \ref{4.2}. 
Suppose that the ideal $I$ 
has a simple factor of order $1$. If the ideal $I$ does not have a simple factor of the form $(x, y^{\ell})$ for 
some $1 \leq \ell \leq r-1$, then there exists $1 \leq k \leq r-1$ such that the module $M_{k}$ is 
indecomposable integrally closed with $I(M_{k})=I$ by Theorem \ref{4.7}. 
If the ideal $I$ has all the simple factors of 
the form $(x, y^{\ell})$ for all $1 \leq \ell \leq r-1$, 
then there exists $1 \leq k \leq r+1$ such that the module $M_{k}$ 
is indecomposable integrally closed with $I(M_k)=I$ by Theorem \ref{4.12}. 

When $r=2$ and $xy \notin I$. Then 
$$I=(x^{2}, xy^{b'}, y^{b})$$
where $b>b' \geq 2$. If $b<2b'$, then $I=\bar{(x^{2}, y^{b})}$ is simple. 
If $b \geq 2b'$, then $I=(x, y^{b'})(x,y^{b-b'})$. 
Therefore, in each case, $I$ has no simple factor $(x, y)$. 
Consider 
$$M_{1}=\left\langle
\begin{pmatrix}
x & y^{b'} & y & 0 \\
0 & 0 & x & y^{b-1}
\end{pmatrix}
\right\rangle. 
$$
Then $I_1(M_1)=(x,y)$. If $M_{1}$ is decomposable, $(x, y) \mid I$ by Observation \ref{4.1}. 
This is a contradiction. Thus, $M_{1}$ is indecomposable. This completes the proof. 
\end{proof} 

When a given monomial ideal $I$ is integrally closed of $\ord(I) = 2$ and $xy \in I$, that is, 
$$I=(x^{a}, xy, y^{b})=(x^{a-1}, y)(x, y^{b-1}) $$
where $a, b \geq 2$, we do not know whether or not there exists an indecomposable integrally closed $R$-module
$M$ with $I(M)=I$. It would be nice to know this remaining case. 

Moreover, it would be interesting to know whether or not 
Theorem \ref{main} holds true for any $\fkm$-primary complete ideal. 
One can ask the following. 

\begin{Question}
{\rm 
For any $\fkm$-primary complete (not necessarily monomial) ideal $I$ of $\ord(I) \geq 3$ 
in a two-dimensional regular local ring $(R, \fkm)$, can we find an indecomposable integrally closed $R$-module $M$ 
of rank $2$ with $I(M)=I$?
}
\end{Question}

It would be also interesting to study the associated module of rank bigger than $2$ in the sense that for any given complete ideal $I$ in $R$ of $\ord(I)=r \geq 3$ and any integer $2 \leq e \leq r-1$, 
can we construct indecomposable integrally closed modules of rank $e$ whose first Fitting ideal is $I$? 

We close the article with some examples to illustrate our results. 

\begin{Example}\label{5.2}
{\rm 
Let $I=(x^{5}, x^{4}y^{2}, x^{3}y^{3}, x^{2}y^{4}, xy^{6}, y^{7})$. Then the Newton polyhedron ${\rm NP}(I)$ 
is given in Figure \ref{fig5.2}. The set of its vertices is 
$$\{(5,0), (2,4), (0,7)\}$$ 
which is denoted by dots in 
Figure \ref{fig5.2}. 
Thus, the Zariski decomposition of $I$ is 
$$I=\bar{(x^{2}, y^{3})} \cdot \bar{(x^{3}, y^{4})}, $$ 
and, hence, $I$ has no simple factor of order $1$. 
Therefore, the associated modules $M_{1}, M_{2}, M_{3}$ and  $M_{4}$ are (non-isomorphic) 
indecomposable integrally closed modules with 
the first Fitting ideal $I$ by Theorem \ref{4.2}. 
}
\end{Example}

\begin{figure}[h]

\begin{multicols}{2}
\includegraphics[width=6.5cm, clip]{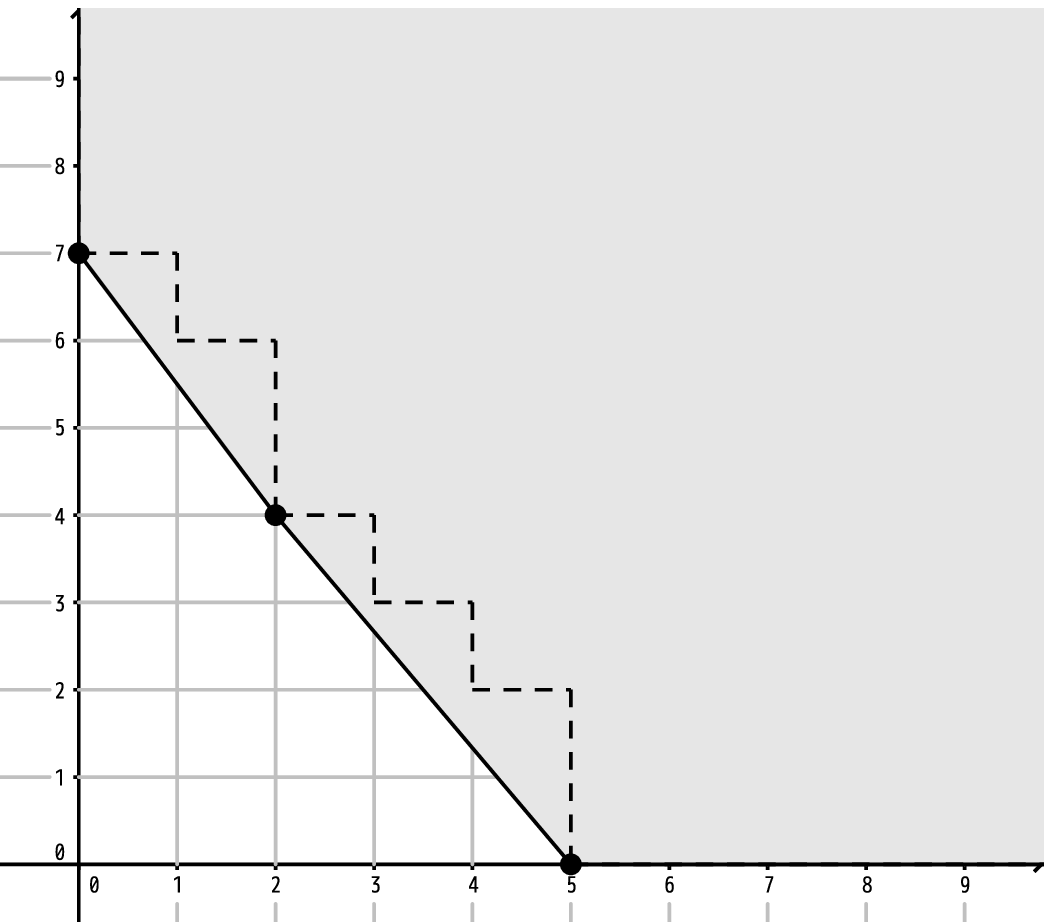}
\caption{Example \ref{5.2} \label{fig5.2}}

\includegraphics[width=6.5cm, clip]{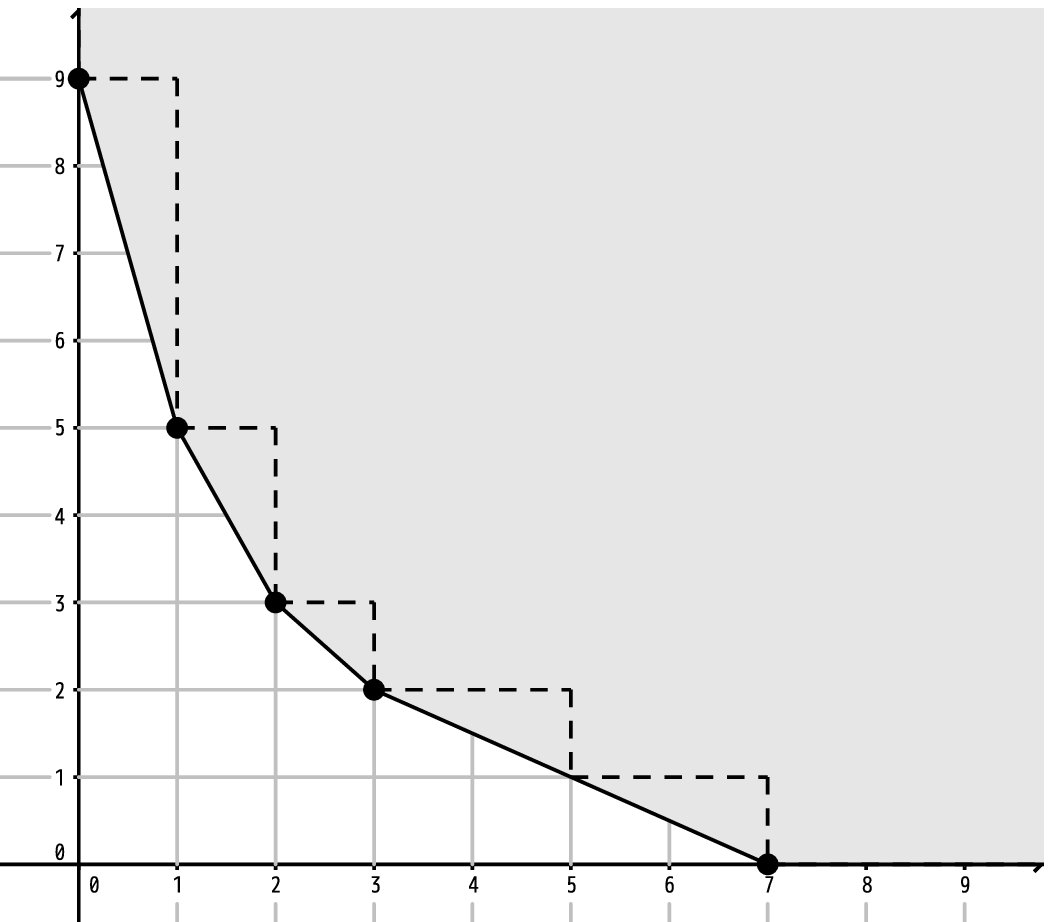}
\caption{Example \ref{5.3} \label{fig5.3}}

\end{multicols}
\end{figure}

\begin{Example}\label{5.3}
{\rm 
Let $I=(x^{7},x^{5}y,x^{3}y^{2}, x^{2}y^{3}, xy^{5},y^{9})$. 
Then the Newton polyhedron ${\rm NP}(I)$ is given in Figure \ref{fig5.3}. The set of its vertices is
$$\{ (7,0),(3,2),(2,3),(1,5),(0,9) \}$$ 
which is denoted by dots in Figure \ref{fig5.3}. 
Thus, the Zariski decomposition of $I$ is 
$$I=(x,y)(x,y^{2})(x,y^{4})(x^{2},y)^{2}. $$ 
Hence, $(x,y)(x,y^{2}) \mid I$ and $(x,y^{3}) \nmid I$. 
Therefore, the associated module $M_{3}$ is indecomposable integrally closed with $I(M_{3})=I$ 
by Observation \ref{4.3}. 
}
\end{Example}

\begin{Example}
{\rm 
Let $\fka=(x,y)(x,y^{2})(x,y^{3})(x,y^{4})$. Consider 
$$
\fkb_{1}=(x^{2}, y), \ 
\fkb_{2}=(x,y^{3}), \ \text{and} \ 
\fkb_{3}=(x,y^{5}). 
$$ 
\begin{enumerate}
\item Let $I=\fka \fkb_{1}$. Then $\ord(I)=5$, and one can apply Theorem \ref{4.9} as $r=5$ and $\alpha=2$. 
Thus, the associated module $M_{5}$ is indecomposable integrally closed with $I(M_{5})=I$. 
\item Let $I=\fka \fkb_{2}$. Then $\ord(I)=5$, and one can apply Theorem \ref{4.11} as $r=5$ and $\beta=3$. 
Thus, the associated module $M_{5}$ is indecomposable integrally closed with $I(M_{5})=I$. 
\item Let $I=\fka \fkb_{3}$. Then $\ord(I)=5$, and one can apply Theorem \ref{4.11} as $r=\beta=5$. 
Thus, the associated module 
$M_{6}$ is indecomposable integrally closed with $I(M_{6})=I$. 
\end{enumerate}
}
\end{Example}

\begin{Example}
{\rm 
Let $I=\fkm^{r}$ where $r \geq 3$. Then $\ord(I)=r$, $(x,y) \mid I$ and $(x,y^{2}) \nmid I$. 
Thus, the associated module 
$M_{2}$ is indecomposable integrally closed with $I(M_{2})=I$ by Observation \ref{4.3}. 
Also, the associated module $M_{r-2}$ is indecomposable integrally closed 
with $I(M_{r-2})=I$ by Proposition \ref{4.5}. 
In fact, by Observation \ref{4.1}, one can see that for any $2 \leq k \leq r-2$, the associated module 
$M_{k}$ is indecomposable integrally closed with $I(M_{k})=I$. 
Moreover, one can show that the module $M_{1}$ is also indecomposable 
integrally closed with $I(M_{1})=I$. 
}
\end{Example}

\begin{proof}
We show the indecomposability of $M_{1}$ by using Buchsbaum-Rim multiplicities. 
Consider a submodule of $M_{1}$: 
$$N=\left\langle \begin{pmatrix}
x^{r-1} & y & 0 \\
0 & x & y^{r-1}
\end{pmatrix}
\right\rangle. $$
Since $I(N) \subset I=I(M)$ is a reduction and $\mu_{R}(N)=3$, 
$N$ is a minimal reduction of $M_{1}$. Thus, we have the equality 
$$e(F/M_{1})=e(F/N)$$ 
for Buchsbaum-Rim 
multiplicities (see \cite[Proposition 3.8]{Ko} for instance). 
Moreover, since $\tilde{N}$ is a parameter matrix in the sense of \cite{BR}, we have the
following equalities.  
$$e(F/M_{1})=e(F/N)=\ell_R(F/N)=\ell_{R}(R/I(N))=\ell_{R}(R/(x^{r}, x^{r-1}y^{r-1}, y^{r}))=r^{2}-1. $$
The second equality follows from \cite[Corollary 4.5]{BR}, and the third one from \cite[2.10]{BV}. 
See also \cite[Theorem 1.3 (2)]{HH}. 
  
Suppose that $M_{1}$ is decomposable. Then $M_{1} \cong \fkm \oplus \fkm^{r-1}$ by Observation \ref{4.1}. 
Consider a submodule of $\fkm \oplus \fkm^{r-1}$: 
$$N'=\left\langle \begin{pmatrix}
x & y & 0 \\
0 & x^{r-1} & y^{r-1}
\end{pmatrix}
\right\rangle. $$
Then $N'$ is a minimal reduction of $\fkm \oplus \fkm^{r-1}$, and we get the following equalities.  
$$e(R/\fkm \oplus R/\fkm^{r-1})=e(F/N')=\ell_{R}(R/I(N'))=\ell_{R}(R/(x^{r}, xy^{r-1}, y^{r}))
=r^{2}-r+1. $$
This contradicts to the assumption $r \geq 3$. 
Thus, we get the indecomposability of $M_{1}$. \end{proof}

\section*{Acknowledgments}
The author would like to thank Professors Shiro Goto, Yukio Nakamura and all the members of the seminar at 
Meiji University for their encouragement and many helpful suggestions at the early stage of this research. 
He would also like to thank Professor Ngo Viet Trung for informing him some papers on the decomposition of 
complete monomial ideals.



\end{document}